\documentclass[a4paper,10pt]{article} \usepackage{amsthm} \usepackage{graphicx, amssymb}  \usepackage{amsmath} \usepackage{mathrsfs} 
\newtheorem{prop}{Proposition}[section]
\newtheorem{thm} [prop]{Theorem}

 \newtheorem{lemma} [prop]{Lemma}

\theoremstyle{definition}
\newtheorem*{ack}{Acknowledgements}

\DeclareMathOperator{\Out}{Out}
\DeclareMathOperator{\Inn}{Inn}
\DeclareMathOperator{\Aut}{Aut}
\DeclareMathOperator{\fix}{fix}

\renewcommand\leq{\leqslant} 
\renewcommand\geq{\geqslant} 

\title{The base size of a primitive diagonal group}
\author{Joanna B. Fawcett}

\begin{document}

\maketitle

\begin{abstract}
A base $\mathscr{B}$ for a finite permutation group $G$ acting on a set $\Omega$ is a subset of $\Omega$ with the property that only the identity of $G$ can fix every point of $\mathscr{B}$. We prove that a primitive diagonal group $G$ has a base of size 2 unless the top group of $G$ is the alternating or symmetric group acting naturally, in which case a tight bound for the minimal base size of $G$ is given. This bound also satisfies a well-known conjecture of Pyber. Moreover, we prove that if the top group of $G$ does not contain the alternating group, then the proportion of pairs of points that are bases for $G$ tends to 1 as $|G|$ tends to infinity. A similar result for the case when the degree of the top group is fixed is given.
\end{abstract}

\section{Introduction}
\label{Introduction}

Let $G$ be a finite permutation group acting on a set $\Omega$. A \textit{base} $\mathscr{B}$ for $G$ is a non-empty subset of $\Omega$ whose pointwise stabiliser is trivial. The \textit{base size} of $G$ is the  minimal cardinality of a base for $G$, and we denote this by $b(G)$.  Bases have been very useful in group theory, both theoretically in bounding the size of a primitive permutation group (e.g. \cite{bab1}) and computationally (surveyed in \cite{ser2}). Accordingly, much research has been done on bounding the  base size of a primitive permutation group $G$ (e.g. \cite{lie}). 

Recently, it has been proved in \cite{ck,gssh,lsh1,bur,bgs,bls,bow} that if $G$ is a finite almost simple primitive permutation group, then $b(G)\leq 7$ unless the action of $G$ is standard, in which case the base size is unbounded in general. ($G$ has a \textit{standard action} if $G$ either has socle $A_m$ and acts on the set of  $k$-subsets or partitions of $\{1,\dots,m\}$, or is a classical group that acts primitively on an orbit of subspaces of its natural module.) This was conjectured to be the case by Cameron \cite{cam}. In fact, it is proved in \cite{bgs} that if $G$ is $A_m$ or $S_m$ acting primitively on a set of size $n$, then $b(G)=2$ unless the action of $G$ is standard or $G$ is one of 13 listed exceptions. Together with the work of J. James  \cite{ja}, this classifies the primitive actions of $S_m$ and $A_m$ with base size 2. A similar result for primitive actions of almost simple classical groups is forthcoming in \cite{bgs2}. With the goal in mind of a theorem classifying which primitive permutation groups admit a base of size 2, we must therefore consider the other types of primitive permutation groups as classified by the O'Nan-Scott Theorem \cite{lps}. These types broadly consist of diagonal groups, twisted wreath products, wreath products, and affine groups. In this paper, we focus on groups of diagonal type. These primitive permutation groups are not often studied, but they are  important, especially to the base size 2 problem. This is because the base size of a primitive diagonal group behaves much like the base size of an almost simple primitive permutation group, in that the base size is either 2, or,  for several explicitly given classes of groups, can be unbounded.

Let $T$ be a finite non-abelian simple group, and let $k$ be an integer that is at least 2. A group of diagonal type $G$  with socle $T^k$ acts primitively on a set $\Omega(k,T)$ with degree $|T|^{k-1}$ and is a (not necessarily split) extension of $T^k$ by a subgroup of $\Out(T)\times S_k$; precise definitions will be given in Section \ref{Preliminaries}. The permutation group induced from the conjugation action of $G$  on the $k$ factors of $T^k$ is called the top group of $G$ and is denoted by $P_G$. The group $P_G$ is either primitive in its action on $k$ points or trivial when $k=2$, and it plays a large part in determining the base size of $G$. Observe that if the top group $P_G$ does not contain the alternating group $A_k$, then we necessarily have $k\geq 5$ since $A_k$ and $S_k$ are the only primitive permutation groups of degree $k$ when $k<5$ (and the only permutation groups when $k=2$).

\begin{thm}
\label{small top}
Let $G$ be a group of diagonal type with socle $T^k$ for some finite non-abelian simple group $T$. If the top group $P_G$ is not the alternating group $A_k$ or the symmetric group $S_k$, then $b(G)=2$.
\end{thm}

This is the best result we could hope for since a group of diagonal type never has a base of size 1. The  proof of Theorem \ref{small top} is constructive, though it depends  on a non-constructive result in \cite{s} which determines exactly when a primitive permutation group has a regular orbit on the power set of the domain of its  action. Note that Gluck, Seress and Shalev \cite{gssh} used \cite{s} to construct a base of size 3 for a group of diagonal type whose top group is neither alternating, symmetric nor of degree less than 32.

The situation is markedly different, however, when the top group $P_G$ is either the alternating group $A_k$ or the symmetric group $S_k$.  Observe that groups of diagonal type can be constructed for any finite non-abelian simple group $T$ and for arbitrarily large $k$.

\begin{thm}
\label{base size for alt}
Let $G$ be a group of diagonal type with socle $T^k$ for some finite non-abelian simple group $T$ where the top group $P_G$ contains the alternating group $A_k$. If $k\geq 3$ then
$$ b(G) = \left\lceil \frac{\log{\ k\ }}{\log{|T|}} \right\rceil +a_G$$
where $a_G\in\{1,2\}$ and $a_G=1$ if $|T|^l<k\leq |T|^{l}+|T|-1$ for some positive integer $l$. If $k=2$, then $b(G)=3$ when $P_G=1$, and $b(G)\in\{3,4\}$ otherwise.
\end{thm}

We will see in Proposition \ref{lower bound for alt} that if either $k=|T|$, or if $\Inn(T)^k\rtimes S_k\leq G$ and $k$ is $|T|^l$ or $|T|^l-1$ for some positive integer $l$, then $a_G=2$. Also, we give examples  when $k=2$ of two groups $G$ with $b(G)=3$ and two groups $G$ with $b(G)=4$ (see the end of Section \ref{Base sizes for diagonal type groups}). Thus the bound of Theorem \ref{base size for alt} is essentially best possible. However, it remains unclear precisely when the two possibilities occur. In particular, we do not know when $b(G)=2$, though $2<k<|T|$ is a necessary condition.

 Theorems \ref{small top} and \ref{base size for alt} also allow us to prove a well-known conjecture of Pyber in the case of diagonal type groups. For any finite permutation group $G$ of degree $n$, it is easy to see that $\lceil\log{|G|}/\log{n}\rceil\leq b(G)$; simply show that $|G|\leq n^{b(G)}$ by considering an appropriate chain of pointwise stabilisers of base elements. Pyber \cite{py1} conjectured that there exists an absolute constant $c$ for which the base size of a primitive permutation group $G$ of degree $n$ is at most $c\log{|G|}/\log{n}$. For example,  almost simple groups with non-standard primitive actions satisfy Pyber's conjecture because their base sizes are bounded above by an absolute constant  \cite{ck,gssh,lsh1}, and Benbenishty \cite{ben} has verified the conjecture for standard actions of almost simple groups. Moreover, soluble primitive permutation groups satisfy Pyber's conjecture  \cite{ser1}, as do certain other affine primitive permutation groups  \cite{glm,lsh2}.

\begin{thm}
\label{pyber}
Let $G$ be a group of diagonal type. Then $G$ satisfies Pyber's conjecture. In fact,
$$b(G)\leq \left\lceil \frac{\log{|G|}}{\log{n}}\right\rceil+2$$
where $n$ is the degree of $G$.
\end{thm}

We remark that although mention is made in \cite{lsh2} of a forthcoming paper by Seress in which Pyber's conjecture is proved for various primitive permutation groups including  groups of diagonal type, this paper did not appear. Moreover, in Gluck, Seress and Shalev \cite{gssh}, a base for groups of diagonal type is constructed and it is claimed there that the argument can be improved to construct a base of size $\lceil \log{|G|}/\log{n}\rceil+3$ (where $G$ is a group of diagonal type with degree $n$), but the details of the proof of this weaker result are not given.

Now we consider the probabilistic side of the theory. The result of  \cite{ck,gssh,lsh1,bur,bgs,bls,bow} that $b(G)\leq 7$  when $G$ is an almost simple primitive permutation group with a non-standard action actually has a stronger form. Cameron and Kantor \cite{ck} conjectured  that for such groups $G$ there exists an absolute constant $c$ with the property that the probability that a random $c$-tuple of points is a base for $G$ tends to 1 as the order of $G$ tends to infinity. In the same paper, Cameron and Kantor proved that their conjecture is true with $c=2$ when the socle of $G$ is alternating. Liebeck and Shalev \cite{lsh1}  then proved the general conjecture for some undetermined constant $c$ by using  \cite{gssh} and counting fixed points of elements. The constant $c=6$ was finally established through work in \cite{lsh3,bls}. We have an analogous result for groups of diagonal type. The proof uses the method of counting fixed points of elements as in \cite{lsh1}.

\begin{thm}
\label{diag prob}
Let $G$ be a group of diagonal type with socle $T^k$ for some finite non-abelian simple group $T$, and suppose that the top group $P_G$ does not contain the alternating group $A_k$. Then the proportion of pairs of points from $\Omega(k,T)$ that are bases for $G$ tends to 1 as $|G|\to \infty$.
\end{thm}

Similarly, we have a partial result that includes the case when the top group $P_G$ contains the alternating group $A_k$. One consequence of this result is that for any fixed $k$ at least 5, there are only finitely many groups of diagonal type with a degree $k$ top group which do not have base size 2.

\begin{thm}
\label{diag prob k fixed}
Let $G$ be a group of diagonal type with socle $T^k$ for some finite non-abelian simple group $T$ where $k\geq 5$. The proportion of pairs of points from $\Omega(k,T)$ that are bases for $G$ tends to 1 if $k$ is fixed as $|G|\to \infty$.
\end{thm}

Explicitly, Theorems \ref{diag prob} and \ref{diag prob k fixed} say the following. Let $\delta$ denote either the symbol $\infty$, or an integer that is at least 5. If $\delta=\infty$, let $\mathscr{D}_\delta$ be the collection of those groups of diagonal type whose top group is not alternating or symmetric, and if $\delta$ is an integer, let $\mathscr{D}_\delta$ be the collection of those groups of diagonal type whose top group has degree $\delta$. For each $\delta$, and for each $G\in\mathscr{D}_\delta$ with socle $T^k$ (where $k=\delta$ if $\delta$ is an integer), let $n_\delta(G)$ denote the proportion of ordered pairs $(\omega_1,\omega_2)$ in $\Omega(k,T)^2$ for which $\{\omega_1,\omega_2\}$ is a base for $G$. Fix $\delta$. Then for every $\varepsilon > 0$, there exists a natural number $N$ such that $n_\delta(G)>1-\varepsilon$ for every $G\in \mathscr{D}_\delta$ satisfying $|G|>N$.

This paper is organised as follows. Section \ref{Preliminaries} gives some basic notation and describes the groups of diagonal type in detail. Theorems \ref{small top}, \ref{base size for alt} and \ref{pyber} are then proved in Section \ref{Base sizes for diagonal type groups}: Theorem \ref{small top} follows from Propositions \ref{small top k>32} and \ref{small top k<32}, while Theorem \ref{base size for alt} essentially follows from Propositions \ref{diagonal base size bound}, \ref{small k} and \ref{lower bound for alt}. The proof of Theorem \ref{diag prob} will take up most of Section \ref{Probabilistic results}, and the proof of Theorem \ref{diag prob k fixed} comes at the end of that section. Note that Sections \ref{Base sizes for diagonal type groups} and \ref{Probabilistic results} are essentially independent of each other. Note also that most of the results  presented in this paper depend upon the classification of the finite simple groups.

\section{Preliminaries}
\label{Preliminaries}

In this paper, all groups are finite and all actions and group homomorphisms are performed on the right. Note that the CFSG refers to the classification of the finite simple groups and that the notation used to denote the finite simple groups is consistent with that of \cite{kl}. 

First we have some basic notation. Let $X$ and $Y$ be groups. We denote the semidirect product of $X$ and $Y$ by $X\rtimes Y$; note that under this notation, $X$ is a normal subgroup of $ X\rtimes Y$, $Y$ acts on $X$, and $(x_1,y_1^{-1})(x_2,y_2)=(x_1x_2^{y_1},y_1^{-1}y_2)$ for all $(x_1,y_1^{-1})$ and $(x_2,y_2)$ in $X\rtimes Y$. If $Y$ acts on $[m]:=\{1,\ldots,m\}$, then $Y$ acts on $X^m$ by permuting the coordinates; that is, the element $y^{-1}$ of $Y$ maps $(x_1,\ldots,x_m)$ to $(x_{1^{y}},\ldots,x_{m^{y}})$
for all $(x_1,\ldots,x_m)\in X^m$. This action defines the wreath product  $X^m\rtimes Y$, which we denote by $X\wr_m Y$. Moreover, if $X$ is a permutation group on a set $\Omega$, then $X\wr_m Y$ acts on $\Omega^m$ by sending $(\omega_1,\ldots,\omega_m)$ to $(\omega_{1^{y}}^{x_{1^{y}}},\ldots,\omega_{m^{y}}^{x_{m^{y}}})$ for each $(x_1,\ldots,x_m)y^{-1}\in X\wr_m Y$. This is called the \textit{product action}. As is standard, we denote the stabiliser in $X$ of the point $\omega\in \Omega$  by $X_\omega$, the conjugacy class of $x\in X$ by $x^X$, and the set of right cosets of the subgroup $Y$ of $X$ by $(X:Y)$. Note that $A_m$ and $S_m$ respectively denote the alternating group and the symmetric group on the set $[m]$. Also, if $x,y\in X$, then $[x,y]=x^{-1}y^{-1}xy$, and if $\alpha\in \Aut(X)$, we write $\overline{\alpha}$ for the coset $\alpha \Inn(X)$ in the outer automorphism group $\Out(X)$. Lastly, the function $\log{x}$ denotes the natural logarithm unless otherwise specified.

The following definitions for groups of diagonal type can be found in \cite{lps}. For an integer $k\geq 2$ and a finite non-abelian simple group $T$, we define
$$
\begin{array}{rl}
 W(k,T)\!\!\!\!&:=\{(\alpha_1,\ldots,\alpha_k)\pi\in \Aut(T)\wr_k S_k:\overline{\alpha}_1=\overline{\alpha}_i \ \mbox{for all}\ i\},\\
D(k,T)\!\!\!\!&:=\{(\alpha,\ldots,\alpha)\pi\in \Aut(T)\wr_k S_k\},\\
\Omega(k,T)\!\!\!\!&:=(W(k,T):D(k,T)),\\
A(k,T)\!\!\!\!&:=W(k,T)\cap \Aut(T)^k.\\
\end{array}
$$
 Note that $W(k,T)=A(k,T)\rtimes S_k$ and that $W(k,T)$ is an extension of $T^k$ by $\Out(T)\times S_k$. Moreover, $W(k,T)$ acts faithfully on the right coset space $\Omega(k,T)$ since $\Inn(T)^k$ is the unique minimal normal subgroup of $W(k,T)$.

We say that a group $G$ has \textit{diagonal type} if there exists an integer $k$ and a finite non-abelian simple group $T$ such that $\Inn(T)^k\leq G\leq W(k,T)$ and $G$ acts primitively on $\Omega(k,T)$. Any such $G$ has socle $T^k$ and degree $n:=|T|^{k-1}$. Let $G$ be a subgroup of $W(k,T)$ containing $\Inn(T)^k$, and let $P_G$ denote the subgroup of $S_k$ consisting of those $\pi\in S_k$ for which there exists $(\alpha_1,\ldots,\alpha_k)\in A(k,T)$ such that $(\alpha_1,\ldots,\alpha_k)\pi\in G$. Then $G$ is a group of diagonal type if and only if either (i) $P_G$ is primitive on $[k]$,  or (ii)  $k=2$ and $P_G=\{1\}$ (see \cite[Theorem 4.5A]{dix}). In particular, $W(k,T)$ is a group of diagonal type.  Note that $P_G$ is permutation isomorphic to the image of the action of $G$ on $\{T_1,\ldots, T_k\}$ by conjugation, where $T_{i}$ is the $i$-th direct factor of $\Inn(T)^k$, since for any $w:=(\alpha_1,\ldots,\alpha_k)\pi\in W(k,T)$, we have $w^{-1}T_iw=T_{i\pi}$ for all $i\in[k]$. The group $P_G$ is referred to as the \textit{top group} of $G$. So long as the context prevents any confusion, we write $D$, $W$ and $\Omega$ for $D(k,T)$, $W(k,T)$ and $\Omega(k,T)$ respectively.
  
Let us briefly examine $\Omega$. Its elements have the form $\omega:=D(\alpha_1,\ldots,\alpha_k)\pi$ for some $(\alpha_1,\ldots,\alpha_k)\pi\in W$. Now $(\alpha_i,\ldots,\alpha_i)\pi\in D(k,T)$ for any $i\in [k]$, so fixing $i$ we see that $\omega=D(\alpha_{i\pi^{-1}}^{-1}\alpha_{1\pi^{-1}},\ldots,1,\ldots,\alpha_{i\pi^{-1}}^{-1}\alpha_{k\pi^{-1}})$
 where 1 is in the $i$-th coordinate. Since $\overline{\alpha}_l=\overline{\alpha}_j$ for all $l$ and $j$, elements of $\Omega$ actually have the form $D(\varphi_{t_1},\ldots,\varphi_{t_k})$, where for each $t\in T$, the map $\varphi_t:T\to T$ is defined to be conjugation by $t$. Moreover, every element of $\Omega$ has $|T|$ representatives in $\Inn(T)^k$, and for each element of $\Omega$, we may choose one coordinate to be any element of $\Inn(T)$ should we wish to do so. In particular, fixing the same coordinate and element of $\Inn(T)$ and allowing all $(k-1)$-tuples with entries in $\Inn(T)$ yields the $|T|^{k-1}$ elements of $\Omega$.

\section{Base sizes for diagonal type groups}
\label{Base sizes for diagonal type groups}

For this section, let $G$ be a group of diagonal type with socle $T^k$ where $T$ is a finite non-abelian simple group.  Note that for $g\in G$ and $\mathscr{B}\subset \Omega$, the set $\mathscr{B}^g$ is a base for $G$ precisely when $\mathscr{B}$ is a base. (Indeed, this is true for any action.) Thus by transitivity there is no loss of generality in restricting our attention to those subsets of $\Omega$ that contain $D$. We begin by determining the pointwise stabiliser in $G$ of any two element subset of $\Omega$ containing $D$.

\begin{lemma} 
\label{pwise stab}
Let $\omega:=D(\varphi_{t_1},\ldots,\varphi_{t_k})\in\Omega$ and write $t^{i,j}$ for $t_i^{-1}t_j$. Then for any $j_0\in[k]$, we have
$G_\omega\cap D=\{(\alpha,\ldots,\alpha)\pi\in G: t^{i,j_0}\alpha=t^{i\pi,j_0\pi} \mbox{\ for all}\ i\}.$
\end{lemma}

\begin{proof}
Fix $j_0\in [k]$. Then $(\alpha,\ldots,\alpha)\pi\in G$ fixes $\omega$ if and only if $\varphi_{t_{j_0}}\alpha\varphi_{t_{j_0\pi}}^{-1}=\varphi_{t_i}\alpha\varphi_{t_{i\pi}}^{-1}$ for all $i$. This is equivalent to $\varphi_{t^{i,j_0}}\alpha=\alpha\varphi_{t^{i\pi,j_0\pi}}$ for all $i$. Evaluating this last expression at $t$ for each $t\in T$, we see that it is equivalent to the statement that $(t^{i,j_0}\alpha) (t^{i\pi,j_0\pi})^{-1}$ centralises $t\alpha$ for all $t\in T$. Since the centre of $T$ is trivial, the proof is complete.
\end{proof}

 Lemma \ref{pwise stab} then has the following useful, easy corollary.

\begin{lemma}
\label{baby}
Suppose that $(\alpha,\ldots,\alpha)\pi\in G$ fixes $D(\varphi_{t_1},\ldots,\varphi_{t_k})\in \Omega$. If there exists $j_0$ for which $t_{j_0}$ and $t_{j_0\pi}$ are trivial, then $t_i\alpha=t_{i\pi}$ for all $i$.
\end{lemma}

Lemma \ref{pwise stab} motivates the following notation. For $\omega:=D(\varphi_{t_1},\ldots,\varphi_{t_k})\in\Omega$, let $\mathcal{O}_\omega$ denote the $k\times k$ matrix whose $(i,j)$-th entry is the order of $t^{i,j}=t_i^{-1}t_j$.  If $(\alpha,\ldots,\alpha)\pi\in G_\omega$, then since $t^{i,j_0}\alpha=t^{i\pi,j_0\pi}$ for all $i$ for any fixed $j_0$ by Lemma \ref{pwise stab}, the $j_0\pi$-th column of $\mathcal O_\omega$ must be a permutation of the entries of the $j_0$-th column. Note that $\mathcal O_\omega$ is a symmetric matrix whose diagonal entries are all 1.

Now we prove Theorem \ref{small top} for $k>32$. The proof relies mainly on a theorem of Seress \cite{s} which determines precisely when a regular orbit on the power set of the domain of a primitive action exists; his work is based on work by Cameron, Neumann and Saxl \cite{cns} who proved    using the CFSG that such a regular orbit exists for all but finitely many degrees so long as the action is not the natural action of the alternating or symmetric group. This result, as mentioned in the introduction, was used by Gluck, Seress and Shalev \cite{gssh}  to construct a base of size 3  for a group of diagonal type  whose top group has degree at least 33 and does not contain the alternating group  (and a larger base otherwise). However, the proof of Proposition \ref{small top k>32} below proceeds somewhat differently in order to construct a base of size 2.

\begin{prop}
\label{small top k>32}
 If $A_k\nleq P_G$ and $k>32$, then $b(G)=2$.
\end{prop}

\begin{proof}
Since $P_G$ is primitive and does not contain $A_k$, and also since $k>32$, \cite[Theorem 1]{s} implies that $[k]$ can be partitioned into two non-empty subsets $\Delta$ and $\Gamma$ such that the setwise stabiliser of $\Delta$ in $P_G$ is trivial. Since the setwise stabiliser of $\Gamma$ must then also be trivial, we may assume without loss of generality that $|\Delta|\geq |\Gamma|$. Clearly $|\Delta|\geq 4$, so we may partition $\Delta$ into two non-empty subsets $\Delta_1$ and $\Delta_2$ such that neither $|\Delta_1|$ nor $|\Delta_2|$ is $|\Gamma|$. Let $x$ and $y$ be generators for $T$ (which is possible by \cite{ag}), and  define $t_i$ to be $1$ if $i\in \Delta_1$,  $x$ if $i\in\Delta_2$, and $y$ if $  i\in\Gamma$. Let $\omega:=D(\varphi_{t_1},\ldots,\varphi_{t_k})$.

Let $(\alpha,\ldots,\alpha)\pi$ be an element of $G$ fixing $\omega$.  Define a function $g:\{1,\ldots,k\}\to \mathbb{N}$ by mapping $i$ to  the number of entries in column $i$ of $\mathcal{O_\omega}$ that are 1, where $\mathcal{O_\omega}$ is as defined above. Writing $\Delta_3=\Gamma$, we have $g(i)=|\Delta_j|$ if $i\in \Delta_j$. Then $g(i)\neq g(j)$ for all $i\in \Gamma$ and $j\in\Delta$, but $g(i)=g(i\pi)$ for all $i$ since by Lemma \ref{pwise stab}, the entries of column $i\pi$ are a permutation of the entries of column $i$. Hence $\Gamma\pi=\Gamma$, so $\pi$ is the identity. But then for any $i\in\Delta_1$, $t_{i\pi}=t_i=1$, so by Lemma \ref{baby}, $\alpha$ must fix both $x$ and $y$ and is therefore the identity. Thus $\{D,D(\varphi_{t_1},\ldots,\varphi_{t_k})\}$ is a base  for $G$. 
\end{proof}

As for $k$ smaller than 32, we need some more lemmas. This first lemma will also be useful in the case when the top group is alternating or symmetric.

\begin{lemma}
\label{distinct}
Let $t_1,\ldots,t_k$ denote elements of $T$ such that at least two of the $t_i$ are trivial, at least one is non-trivial, and if $t_i$ and $t_j$ are non-trivial and $i\neq j$ then $t_i\neq t_j$. If $(\alpha,\ldots,\alpha)\pi\in G$ fixes $D(\varphi_{t_1},\ldots,\varphi_{t_k})$, then $t_i\alpha=t_{i\pi}$ for all $i$.
\end{lemma}

\begin{proof}
Let $\omega:=D(\varphi_{t_1},\ldots,\varphi_{t_k})$, and let $r_i$ denote the order of $t_i$. Also, let $m$ be the number of non-trivial $t_i$. To begin, assume that $t_i\neq 1$ if $i\in[m]$ and that $t_i=1$ otherwise. Then
$$\mathcal O_\omega= \left( \begin{array}{cc}
A & B \\
B^T & 1_{k-m} \\
 \end{array} \right),$$
where $A$ is a symmetric $m\times m$ matrix whose diagonal entries are 1 and whose remaining entries are integers at least 2, $B$ is an $m\times (k-m)$ matrix with $i$-th row $(r_i,\ldots,r_i)$, and $1_{k-m}$ is a $(k-m)\times (k-m)$ matrix in which every entry is 1. Since $k-m\geq 2$, columns $m+1$ through $k$ each have at least two entries that are 1, and these are the only such columns; hence $\pi$ must permute these columns, which implies that $t_{i\pi}=1$ for $i\geq m+1$. The result then follows from Lemma \ref{baby}. The proof of the general case is essentially the same since then the entries in each column of $\mathcal O_\omega$ will be a permutation of the entries in a column of the matrix above.
\end{proof}

A result of Malle, Saxl and Weigel   \cite{msw} states that every finite non-abelian simple group other than $U_3(3)$ is generated by an involution and a strongly real element. Since $U_3(3)$ is generated by an involution and an element of order 6 by \cite{con}, it follows that every finite non-abelian simple group is generated by two elements, one of which can be taken to be an involution. Since two involutions generate a dihedral group, the two generators must have different orders. This makes the next two lemmas useful. For $x,y\in T$, let $T(x,y)$ denote the set of non-trivial elements of $T$ whose orders are different to the orders of $x$ and $y$.

\begin{lemma}
\label{top base 2}
Suppose that $T=\langle x,y\rangle$ where $x$ and $y$ have different orders, and suppose that $k\geq 4$ and $P_G\neq S_k$. If $P_G$ has base size at most $|T(x,y)|+2$ in its action on $[k]$, then $b(G)=2$.
\end{lemma}

\begin{proof}
We may assume without loss of generality that $\{1,2,\ldots,m\}$ is a base of minimal size for $P_G$. Since $P_G$ is primitive and $P_G\neq S_k$, it follows that $P_G$ contains no transpositions; thus we may conclude that $k\geq m+2$. Let $t_1:=x$, $t_2:=y$, $t_i:=1$ for $\mbox{max}\{3,m+1\}\leq i\leq k$, and when $m\geq 3$, choose $t_3,\ldots,t_m$  to be distinct elements of $T(x,y)$. Suppose that $(\alpha,\ldots,\alpha)\pi\in G$ fixes $D(\varphi_{t_1},\ldots,\varphi_{t_k})$. Then the conditions of Lemma \ref{distinct} are met, so $t_i\alpha=t_{i\pi}$ for all $i$. But $\alpha$ preserves order, so  $\alpha$ fixes $x$ and $y$ and is therefore the identity. Then since the $t_i$ are distinct for $i\in[m]$, $\pi$ is the identity on $[m]$. Hence $\pi$ is the identity, and it follows that $\{D,D(\varphi_{t_1},\ldots,\varphi_{t_k})\}$ is a base  for $G$.
\end{proof}

There is a classical result of Bochert \cite{boc} from the nineteenth century which states that every primitive permutation group of degree $k$ that does not contain $A_k$ has a base of size at most $k/2$ (see \cite[Theorem 3.3B]{dix} for a proof). This makes the following consequence of Lemma \ref{top base 2} possible. 

\begin{lemma}
\label{index}
Suppose that $T=\langle x,y\rangle$ where $x$ and $y$ have different orders, and let $C$ be a non-trivial conjugacy class of $T$ with minimal cardinality. If $A_k\nleq P_G$ and $k\leq 2|C|+4$, then $b(G)=2$. 
\end{lemma}

\begin{proof}
Certainly $|C|\leq |T(x,y)|$ since 3 distinct primes divide $|T|$, while $P_G$ has a base of size at most $k/2$ by Bochert \cite{boc}. Thus the assumption that $k\leq 2|C|+4$ implies that $P_G$ has base size at most $|T(x,y)|+2$ in its action on $[k]$. Note that $k\geq 5$ since $A_k\nleq P_G$ and $P_G$ is primitive. Hence we may apply Lemma \ref{top base 2}.
\end{proof}

\begin{prop}
\label{small top k<32}
 If $A_k\nleq P_G$ and $k\leq 32$, then $b(G)=2$.
\end{prop}

\begin{proof}
By Malle, Saxl and Weigel \cite[Theorem B]{msw}, $T$ is generated by elements $x$ and $y$ with different orders. Let $p(T)$ denote the minimal index of a proper subgroup of $T$. Then by Lemma \ref{index}, $G$ has base size 2 if $32\leq 2p(T)+ 4$, so we may assume that $p(T)\leq 13$. Note that $|T|\leq 13!/2$ since $T$ can be embedded in the alternating group on $p(T)$ points. If $T$ is a classical group of Lie type, then values for $p(T)$ can be found in \cite{maz2,vm1}, and if $T$ is an exceptional group of Lie type, then values for $p(T)$ can be found in \cite{v1,v2,v3}. Of course $p(A_m)=m$, and if $T$ is sporadic (and of order less than $13!/2$), then values for $p(T)$ can be found in \cite{con}. Using these, we see that $T$ must be one of $L_2(7)$, $L_2(8)$, $L_2(11)$, $L_3(3)$, $M_{11}$, $M_{12}$ or $A_m$ for $5\leq m\leq13$. However, it can be seen using \cite{con} that, with the exception of $A_5$, none of these groups has a conjugacy class of size less than 13, and so $b(G)=2$ by Lemma \ref{index}. Lastly, $A_5$ is $(2,3)$-generated and has 24 elements of order 5, while $P_G$ has a base of size at most 32/2 by Bochert \cite{boc}, so $b(G)=2$ by Lemma \ref{top base 2}.
\end{proof}

Together, Propositions \ref{small top k>32} and \ref{small top k<32} imply that $b(G)=2$ when $A_k\nleq P_G$, which establishes Theorem \ref{small top}. Note that Pyber's conjecture (Theorem \ref{pyber}) is therefore true when $A_k\nleq P_G$. 

We now move on to consider those diagonal type groups $G$ for which $P_G$ does contain the alternating group $A_k$. Here it is readily seen that we will not always have base size 2: if $k>|T|$, then every element of $\Omega\setminus\{D\}$ is determined by a $k$-tuple of elements of $T$ whose coordinates contain at least one repeat, and so $W(k,T)$ does not have base size 2. In fact, we will see that $b(G)\neq 2$ when $k\geq |T|$. We begin by constructing a base for $G$.

\begin{prop}
\label{diagonal base size bound}
$G$ has a base of size
$$\left  \lceil\frac{\log{(k-|T|+1)}}{\log{|T|}} \right \rceil +2 $$
if $k>|T|$ and a base of size 3 if $5\leq k\leq |T|$.
\end{prop}

\begin{proof}
Assume that $k\geq 5$. Then $m:=\min{(|T|-1,k-2)}$ is at least 3. Define the positive integer
$$
r:=
\left \{
\begin{array}{ll}
\left  \lceil\frac{\log{(k-|T|+1)}}{\log{|T|}} \right \rceil & \mbox{if} \ \ k> |T|, \\
1 & \mbox{if}\ \ k\leq |T|.
\end{array}
 \right.
$$
For $j$ such that $m<j\leq k$, let $d_{j,0},\ldots, d_{j,r-1}$ denote the first $r$ digits of the base $|T|$ representation of $j-m-1$; this is reasonable since $|T|^{r-1}\leq k-m-1<|T|^r$.  Let $x$ and $y$ be generators for $T$ (by \cite{ag}). Since $|T|$ is divisible by at least 3 distinct primes, we may choose some non-trivial $z$ from $T$ whose order is different to that of $x$ and $y$.  Enumerating the elements of $T$ as $t_0,\ldots,t_{|T|-1}$ where $t_0:=1$, $t_1:=x$, $t_2:=y$ and $t_3:=z$, we may define
$$
u_{i,j}:=
\left \{ 
\begin{array}{ll} 
 t_j & \mbox{if} \ \ i=2\ \ \mbox{and}\ \ 1\leq j\leq m, \\  
x & \mbox{if} \ \ i=3\ \ \mbox{and}\ \ j=1, \\ 
 z & \mbox{if} \ \ i=3\ \ \mbox{and}\ \ j=2, \\
 t_{d_{j,i-3}}& \mbox{if} \ \ 3\leq i\leq r+2\ \ \mbox{and}\ \ m<j\leq k, \\
 1 & \mbox{otherwise.} \
\end{array} 
\right.
$$
For $1\leq i\leq r+2$, let $\omega_i$ denote the element $D(\varphi_{u_{i,1}},\ldots,\varphi_{u_{i,k}})$ of $\Omega$. We claim that $\mathscr{B}:=\{\omega_1,\ldots,\omega_{r+2}\}$ is a base for $G$. Note that $|\mathscr{B}|=r+2$, for if $\omega_i=\omega_{i'}$ for some distinct $i$ and $i'$, then there exists $t\in T$ for which $u_{i,j}=tu_{i',j}$ for all $j$. But then we must have $i,i'\geq 4$, which implies that $t=1$, and so $d_{j,i-3}= d_{j,i'-3}$ for every $j>m$. This is certainly  not the case; for example, take $j=|T|^{i-3}+m+1$.

Since $u_{2,1},\ldots,u_{2,k}$ satisfy the conditions of Lemma \ref{distinct}, we get that $u_{2,j}\alpha=u_{2,j\pi}$ for all $j$, and so $[m]\pi=[m]$. Now $u_{2,3}=z$ has order different to that of $u_{2,1}=x$ and $u_{2,2}=y$, so $3\leq 3\pi\leq m$. Hence $u_{3,3\pi}=1=u_{3,3}$, which implies that $u_{3,j}\alpha=u_{3,j\pi}$ for all $j$  by Lemma \ref{baby}. But $1\pi\leq m$, so $u_{3,1\pi}\in\{x,z,1\}$; this together with the fact that $u_{3,1}\alpha=u_{3,1\pi}$ forces $1\pi=1$. Similarly, $2\pi=2$, but then $u_{2,j}\alpha$ equalling $u_{2,j\pi}$ for $j\in\{1,2\}$ implies that $x\alpha=x$ and $y\alpha=y$, so $\alpha$ is the identity. Moreover, for any $i\geq 4$ we have that $u_{i,1\pi}=u_{i,1}=1$, so it follows from Lemma \ref{baby} that $u_{i,j\pi}=u_{i,j}$ for all $i$ and $j$. In other words, for every $j$, the $j$-th and $j\pi$-th columns of the $(r+2)\times k$ matrix whose $(i,j)$-th entry is $u_{i,j}$ are the same. However, by construction columns $1,\ldots,m$ are distinct from one another, as are columns $m+1,\ldots,k$. Recalling that $[m]\pi=[m]$, it follows that $\pi$ is the identity.
\end{proof}

Note that the CFSG was only used in the proof above to obtain that $T$ is $2$-generated. This assumption can be removed if $k$ is sufficiently larger than $|T|$: let $x_1,\ldots,x_s$ be a set of generators for $T$, and in the construction of $\mathscr{B}$ above, change $x$ to $x_1$, $y$ to $x_2$, and $u_{i+1,2}$ to $x_{i}$ for $3\leq i\leq s$. The proof  remains unchanged until we obtain $x_1\alpha=x_1$ and $x_2\alpha=x_2$. Since $u_{i,1\pi}=u_{i,1}=1$ for $i\geq 4$, Lemma \ref{baby} implies that $u_{i,j}\alpha=u_{i,j\pi}$ for all $i$ and $j$, but $2\pi=2$, so $x_i\alpha=x_i$ for all $i$. The remainder of the proof is the same. To get a crude idea of how large $k$ need be, note that $T$ has a generating set of size at most $\log_2|T|$ (as any finite group does), so we need $\log_2|T|+1$ to be at most $ r+2$ for this argument to work. Hence for $k\geq |T|^{\log_2{|T|}}$, the upper bound on the base size of $G$ in Proposition \ref{diagonal base size bound} is obtained without the CFSG. 

Now we consider small values for $k$. The following will be used when $k=2$.

\begin{lemma}
\label{centralise}
If $\{D,D(\varphi_{t_1},\ldots,\varphi_{t_k})\}$ is a base of size 2 for $G$, then
$\bigcap_{i=1}^kC_T(t_i)=\{1\}.$
\end{lemma}

\begin{proof}
If $t\in\bigcap_{i=1}^kC_T(t_i)$, then $(t_i^{-1}t_1)\varphi_t=t_i^{-1}t_1$ for all $i$. But $(\varphi_t,\ldots,\varphi_t)\in G$, and $(\varphi_t,\ldots,\varphi_t)$ fixes $D(\varphi_{t_1},\ldots,\varphi_{t_k})$  by Lemma \ref{pwise stab}, so $t=1$.
\end{proof}

\begin{prop}
\label{small k}
If $P_G=A_k$, then $b(G)=3$ when $k=2$, and $b(G)=2$ when $k$ is 3 or 4. If $P_G=S_k$, then $b(G)\in\{3,4\}$ when $k=2$, and $b(G)\in\{2,3\}$  when $k$ is 3 or 4.
\end{prop}

\begin{proof}
Let $x$ and $y$ be generators for $T$ (by \cite{ag}). Then  $\{D,D(\varphi_x,1),D(\varphi_y,1)\}$ or $\{D,D(\varphi_x,1),D(\varphi_y,1),D(\varphi_{xy},1)\}$ is a base for $G$ when $P_G$ is 1 or $S_2$ respectively by Lemma \ref{pwise stab}. Moreover, $b(G)\neq 2$ in these cases since $\{D,D(\varphi_{t},1)\}$ is not a base for $G$ for any $t\in T$ by Lemma \ref{centralise}. Let $z$ be a non-trivial element of $T$ with order different to that of $x$ and $y$. Then Lemma \ref{distinct} implies that $\{D,D(\varphi_x,1,1),D(1,\varphi_y,1)\}$ or $\{D,D(\varphi_x,\varphi_z,1,1),D(1,1,\varphi_y,1)\}$ is a base for $G$ when $P_G$ is $S_3$ or $S_4$ respectively. Since the natural action of $A_4$ has base size 2, it follows from Lemma \ref{top base 2} and \cite[Theorem B]{msw} that $b(G)=2$ when $P_G=A_4$. This leaves us with the case $P_G=A_3$. By \cite[Theorem B]{msw}, we may assume that $y$ is an involution. Then a  consideration of the matrix $\mathcal{O}_{D(\varphi_x,\varphi_y,1)}$ shows that $\{D,D(\varphi_x,\varphi_y,1)\}$ is a base for $G$. 
\end{proof}

We are now able to prove Pyber's conjecture for groups of diagonal type. 

\begin{proof}[Proof of Theorem \ref{pyber}]
Let $G$ be a group of diagonal type with socle $T^k$. It is well known that $k^k/e^{k-1}\leq k!$ for any integer $k\geq 2$. If $A_k\leq P_G$, then
$$\left (\frac{k|T|}{e} \right ) ^{k-1}\leq \frac{1}{2} \left(\frac{k^k}{e^{k-1}}\right)|T|^{k-1}\leq |P_G||T|^{k-1}\leq |G_D||T|^{k-1}= |G|,$$
from which we obtain
$$\frac{\log{(k|T|/e) }}{\log{|T|}}  \leq \frac{\log{|G|}}{\log{|T|^{k-1}}} .$$
But $k-|T|+1\leq k|T|/e$, so when $A_k\leq P_G$ and $k>|T|$, Proposition \ref{diagonal base size bound} implies that $G$ satisfies Pyber's conjecture and, in particular, the bound in the statement of Theorem \ref{pyber}. Since $b(G)$ is constant and at most 4 when $A_k\nleq P_G$ or $k\leq |T|$ by Propositions \ref{small top k>32}, \ref{small top k<32},  \ref{diagonal base size bound} and \ref{small k}, and since we always have  $\lceil\log{|G|}/\log{|T|^{k-1}}\rceil\geq 2$, the proof is complete.
\end{proof}

In fact, since $\log{|G|}/\log{|T|^{k-1}}\leq b(G)$ (see the introduction), this proof provides a lower bound for $b(G)$ which is the desired bound of Theorem \ref{base size for alt} if $e|T|^l<k\leq |T|^{l+1}$ for some non-negative integer $l$. However, this can be improved upon. To do so, we need to know more about the structure of $G$.

\begin{lemma}
\label{Ak in G}
 Suppose that $A_k\leq P_G$. If there exists an odd integer $s$ with $1<s\leq k$ such that $s$ is relatively prime to the order of every element of $\Out(T)$, then $\Inn(T)^k \rtimes A_k\leq G$. 
\end{lemma}

\begin{proof}
 If $\pi$ is an $s$-cycle, then $\pi\in A_k\leq P_G$, so $(\alpha,\ldots,\alpha)\pi\in G$ for some $\alpha\in \Aut(T)$ whose image $\overline{\alpha}$ in $\Out(T)$ has order $r$, say. Certainly $(\alpha^r,\ldots,\alpha^r)\pi^r\in G$, but $G$ contains $\Inn(T)^k$, so $\pi^r$  is an element of $G$. Hence $\pi$ is as well. As $\pi$ was an arbitrary $s$-cycle, the group $G$ contains every $s$-cycle. But the $s$-cycles generate $A_k$, so $\Inn(T)^k \rtimes A_k\leq G$. 
\end{proof}

The next result provides the lower bound on $b(G)$ of Theorem \ref{base size for alt}. In fact, several other lower bounds are proved under somewhat specialised conditions; this is done to show that the bounds of Theorem \ref{base size for alt} are essentially best possible.

\begin{prop}
\label{lower bound for alt} 
Suppose that $A_k\leq P_G$, and let $l$ be a positive integer. Suppose that either $k>|T|^l$, or $l=1$ and $k=|T|$, or $\Inn(T)^k\rtimes S_k\leq G$ and  $k$ is $|T|^l$ or $|T|^l-1$. Then $b(G)\geq l+2$. 
\end{prop}

\begin{proof}
Suppose that one of the four assumptions on $k$ and $G$ in the statement of the proposition is true. Then certainly $k\geq |T|-1$, but $|\Out(T)|$ is much smaller than $|T|$ by the CFSG (see Lemma \ref{Out(T)}, for example), so we may take the $s$ of Lemma \ref{Ak in G} to  be $|\Out(T)|+1$ if $|\Out(T)|$ is even and $|\Out(T)|+2$ otherwise. Thus $\Inn(T)^k\rtimes A_k\leq G$. 

For ease of notation, let $\mathscr{C}$ denote the set of the $|T|^l$ columns of length $l$ with entries in $T$, and for $M$  an $l\times m$ matrix  with entries  in $T$, let $\mathscr{C}_M$ denote the subset of $\mathscr{C}$ whose elements are the columns of $M$. Note that $\Aut(T)$ acts naturally on $\mathscr{C}$.  Suppose that the columns of $M$ are pairwise distinct. If $\mathscr{C}_M^\alpha= \mathscr{C}_M$ for some $\alpha\in \Aut(T)$, then $\alpha$ determines a permutation on $[m]$; this we denote by $\pi_{\alpha,M}$. Note that for each row $(t_1,\ldots,t_m)$ of $M$, we have $t_i\alpha=t_{i\pi_{\alpha,M}}$ for all $i$.

Choose $l$ distinct elements $\omega_1,\ldots,\omega_l$ from $\Omega\setminus \{D\}$, and let $\mathscr{B}$ be the set $\{\omega_i:1\leq i\leq l\}$. We must show that $\mathscr{B}$ is not a base for $G_D$. For each $i$, let $(t_{i,1},\ldots,t_{i,k})$ be one of the $|T|$ choices of $k$-tuples of elements in $T$ that correspond to $\omega_i$. Let $B$ be the $l\times k$  matrix whose $(i,j)$-th entry is $t_{i,j}$.  Note that for each $i$, $(t,\ldots,t)\omega_i=\omega_i$ for any $t\in T$. This allows us either to choose  one element from $\mathscr{C}$ to be column $j$ of $B$ for any one $j\in [k]$, or, when an element of $\mathscr{C}$ is not in $\mathscr{C}_B$, to choose any one element from $\mathscr{C}$ to be in $\mathscr{C}\setminus\mathscr{C}_B$ (with appropriate repercussions for the columns of $B$ in either case).

Suppose that $B$ has three identical columns, say $j_1$, $j_2$ and $j_3$. Then $(1,\ldots,1)(j_1\ j_2\ j_3)$ is an element of $G_D$ that fixes $\mathscr{B}$ pointwise, so $\mathscr{B}$ is not a base for $G_D$. Similarly, if $B$ has two pairs of identical columns, then $\mathscr{B}$ is not a base for $G_D$, so we may assume that neither scenario occurs in $B$. In particular, $k\leq |T|^l+1$.

Suppose that $B$ has exactly one pair of repeated columns. By relabelling if necessary, we may assume that the indices of these columns are $k-1$ and $k$. If $\Inn(T)^k\rtimes S_k\leq G$, then clearly $\mathscr{B}$ is not a base for $G_D$. Moreover, suppose that $k=|T|^l+1$ or that  $l=1$ and $k=|T|$. Then $\mathscr{C}_B$ is $\mathscr{C}$ in the former case and  $\mathscr{C}\setminus\{(t)\}$ for some $t\in T$ in the latter. We may assume by the note above that every entry of column $k-1$ is the identity, and therefore the same is true for column $k$. Let $B^*$ be the $l\times (k-2)$ matrix whose $j$-th column is the $j$-th column of $B$ for $1\leq j\leq k-2$. If $k=|T|^l+1$, let $\alpha$ be any non-trivial element of $\Inn(T)$, and if $l=1$ and $k=|T|$, let $\alpha$ be any non-trivial element of $\Inn(T)$ that fixes $t$. Then $\mathscr{C}_{B^*}^\alpha= \mathscr{C}_{B^*}$ in either case. Since the columns of $B^*$ are pairwise distinct by assumption,  the permutation $\pi_{\alpha,B^*}$ on $[k-2]$ exists as defined above. Moreover, $\pi_{\alpha,B^*}$ can be made into an even permutation $\pi$ of $[k]$ by either fixing or interchanging $k-1$ and $k$. Then $(\alpha,\ldots,\alpha)\pi\in G_D$. Since $t_{i,j}\alpha=t_{i,j\pi}$ for all $i$ and $j$, Lemma \ref{pwise stab} implies that $(\alpha,\ldots,\alpha)\pi$ fixes $\mathscr{B}$ pointwise. Thus $\mathscr{B}$ is not a base for $G_D$. 

Hence we may assume that the columns of $B$ are pairwise distinct. Then $k\leq |T|^l$. If $k=|T|^l$, then $\mathscr{C}_B=\mathscr{C}$, and if $k=|T|^l-1$, then $\mathscr{C}_B=\mathscr{C}\setminus \{c\}$ for some $c\in \mathscr{C}$; we may assume that all of the entries of $c$ are the identity. Let $\alpha$ be an element of $\Inn(T)$ for which $\alpha^2\neq 1$. Then $\mathscr{C}_B^\alpha=\mathscr{C}_B$ in either case. Again, since the entries of $B$ are pairwise distinct, we have a permutation $\pi:=\pi_{\alpha,B}$ of $[k]$. Since $t_{i,j}\alpha=t_{i,j\pi}$ for all $i$ and $j$, Lemma \ref{pwise stab} implies that $(\alpha,\ldots,\alpha)\pi\in D$ fixes $\mathscr{B}$ pointwise; hence the non-trivial element $(\alpha,\ldots,\alpha)^2\pi^2$ of $\Inn(T)^k\rtimes A_k$  does so as well, and thus $\mathscr{B}$ is not a base for $G_D$. 
\end{proof}

\begin{proof}[Proof of Theorem \ref{base size for alt}]
By Propositions \ref{diagonal base size bound} and \ref{small k}, we have the desired result if $k\leq |T|$, so we may assume that $|T|^l<k\leq |T|^{l+1}$ for some positive integer $l$. The upper bound of Theorem \ref{base size for alt} is immediate from Proposition \ref{diagonal base size bound}, while the lower bound follows from the fact that $b(G)\geq l+2$ by Proposition \ref{lower bound for alt}. If we also assume that $k\leq |T|^{l}+|T|-1$, then the upper bound of Proposition \ref{diagonal base size bound} is equal to $\lceil\log{k}/\log{|T|}\rceil+1$ since $k>|T|^l$ implies that $k-|T|+1> |T|^{l-1}$, so $a_G=1$ and the proof is complete.
\end{proof}

Note that Proposition \ref{lower bound for alt} provides several infinite classes of groups for which the $a_G$ of Theorem \ref{base size for alt} is 2; namely, $a_G=2$ when $k=|T|$ or when $G$ contains $\Inn(T)^k\rtimes S_k$ and  $k$ is $|T|^l$ or $|T|^l-1$ for any positive integer $l$. Additionally, it can be shown that if $m$ is 5 or 6, then $b(\Inn(A_m)^2\rtimes S_2)=3$ while $b(W(2,A_m))=4$. (GAP \cite{GAP4} was used to verify this when $m=6$.) Thus the bound on the base size of Theorem \ref{base size for alt} is essentially best possible. 

Furthermore, Proposition \ref{lower bound for alt} implies that $b(G)\neq 2$ when $k\geq |T|$, and if $2<k<|T|$, then we know that $b(G)\in\{2,3\}$ by Propositions \ref{diagonal base size bound} and \ref{small k}. At this stage, it remains unclear whether we can determine when $b(G)=2$ more precisely than this. The main difficulty here lies with the possibility of the existence of two groups of diagonal type with the same socle and top group but different base sizes; indeed, none of the methods we have seen so far can distinguish the  base sizes of two such groups. However, as mentioned in the introduction, we will see in Section \ref{Probabilistic results} that for a particular fixed $k$ that is at least 5, there are only finitely many groups of diagonal type with a degree $k$ top group which do not have base size 2.

\section{Probabilistic results}
\label{Probabilistic results}

In this section, let $T$ be a finite non-abelian simple group. The following argument has been made by Liebeck and Shalev \cite{lsh1}. Let $G$ be a transitive permutation group on $\Omega$. Let $Q(G,b)$ denote the proportion of $b$-tuples in $\Omega^b$ that are not (ordered) bases for $G$. If $x\in G$, the proportion of points in $\Omega$ that are fixed by $x$ is 
$|\fix(x)|/|\Omega|$,
so the proportion of $b$-tuples that are fixed by $x$ is  $(|\fix(x)|/|\Omega|)^b$. Moreover, if a $b$-tuple is not a base for $G$, then it is fixed by some element in $G$ of prime order. Let $X$ be the set of elements in $G$ of prime order, and let $x_1,\ldots,x_l$ be a set of representatives for the $G$-conjugacy classes of elements in $X$. Then since $|\fix(x)|/|\Omega|=|G_\omega\cap x^G|/|x^G|$ for any $\omega\in \Omega$ by transitivity, we have
$$ Q(G,b)\leq \sum_{x\in X}\left (\frac{|\fix(x)|}{|\Omega|}\right )^b=\sum_{x\in X}\left (\frac{|G_\omega\cap x^G|}{|x^G|}\right )^b=\sum_{i=1}^l \frac{|G_\omega\cap x_i^G|^b}{|x_i^G|^{b-1}}.$$
In particular, it follows that if  
$$\sum_{i=1}^l \frac{|G_\omega\cap x_i^G|^2|C_G(x_i)|}{|G|}\to 0$$
as $|G|\to \infty$ for some $\omega\in \Omega$, then almost any pair of elements in $\Omega$ forms a base for $G$. Note that we may choose $x_1,\ldots,x_l$ to be elements of $G_\omega$ since $|G_\omega\cap x_i^G|=0$ if no $G$-conjugate of $x_i$ lies in $G_\omega$.

Let $G$ be a group of diagonal type with socle $T^k$. Choose a set $R(G)$ of representatives for the $G$-conjugacy classes of elements in the stabiliser $G_D$ of $D=D(k,T)$ in $G$ which have prime order. Define
$$
\begin{array}{rl}
R_1(G)\!\!\!\!&:=\{(\alpha,\ldots,\alpha)\pi\in R(G) : \pi \ \mbox{is fixed-point-free on} \ [k]\}, \\
R_2(G)\!\!\!\!&:=\{(\alpha,\ldots,\alpha)\pi\in R(G) : \pi=1\},\\
R_3(G)\!\!\!\!&:=\{(\alpha,\ldots,\alpha)\pi\in R(G) : \pi\neq 1 \ \mbox{and} \ i\pi=i \ \mbox{for some} \ i\in [k]\},\\ 
\end{array}
$$
and for $1\leq i\leq 3$, define
$$r_i(G):=\sum_{x\in R_i(G)} \frac{|G_D\cap x^G|^2|C_G(x)|}{|G|}.$$
Thus $Q(G,2)\leq r_1(G)+r_2(G)+r_3(G)$. We write $\overrightarrow{\alpha}$ for the tuple $(\alpha,\ldots,\alpha)$ and $C$ for some absolute constant which need not and will not be determined (though it could be). Such methodology will also apply to another absolute constant $c>1$, though it will be  obvious what $c$ needs to be. 

We need to prove the following three lemmas; for the second, note that $p(T)$ denotes the minimal index of a proper subgroup of $T$.

\begin{lemma}  
\label{r_1(G)}
Let $P$ be a primitive subgroup of $S_k$ that does not contain $A_k$, and let $G:=A(k,T)\rtimes P$. Then 
$$r_1(G)\leq \frac{C}{c^k|T|^{\frac{1}{6}}}$$
for some absolute constants $C$ and $c>1$.
\end{lemma}

\begin{lemma}
\label{r_2(G)} 
Let $P$ be a primitive subgroup of $S_k$ where $k\geq 5$, and let $G:=A(k,T)\rtimes P$. Then 
$$r_2(G)\leq \frac{C}{p(T)^{k-\frac{19}{4}}}$$
for some absolute constant $C$.
\end{lemma}

\begin{lemma}
\label{r_3(G)} 
Let $P$ be a primitive subgroup of $S_k$ that does not contain $A_k$, and let $G:=A(k,T)\rtimes P$. Then 
$$r_3(G)\leq \frac{C}{|T|^{\frac{1}{3}}}\left( \frac{1}{c^k}+\frac{1}{\sqrt{k}} \right)$$
for some absolute constants $C$ and  $c>1$. 
\end{lemma}

In fact, Lemma \ref{r_3(G)} is primarily a consequence of the following more general result, which we record here and prove separately from Lemma \ref{r_3(G)} as it has applications to the probabilistic side of the base size problem for other types of primitive permutation groups, such as the groups of twisted wreath type (in \cite{faw}).

\begin{lemma}
 \label{r_3(G)*}
Let $P$ be a primitive subgroup of $S_k$ that does not contain $A_k$, and let $T$ be a finite non-abelian simple group. Then for some absolute constants $C$ and $c>1$, we have 
$$  \displaystyle\sum_{\pi \in R(P)} \frac{|\pi^P|}{|T|^{k-r_\pi-\frac{5}{3}}}\leq C\left( \frac{1}{c^k}+\frac{1}{\sqrt{k}} \right)$$
where $R(P)$ denotes a set of representatives for the conjugacy classes of elements of prime order in $P$, and $r_\pi$ denotes the number of cycles in the full cycle decomposition of $\pi$ in $S_k$, including fixed points.
\end{lemma}

If we assume that Lemmas \ref{r_1(G)}, \ref{r_2(G)} and \ref{r_3(G)} are true, then Theorem \ref{diag prob} can be proved easily, as we now see. 

\begin{proof}[Proof of Theorem \ref{diag prob}]
By Lemmas \ref{r_1(G)}, \ref{r_2(G)} and \ref{r_3(G)},
$$
Q((A(k,T)\rtimes P_G),2) \leq C\left ( \frac{1}{c^k|T|^{\frac{1}{6}}}+\frac{1}{p(T)^{k-\frac{19}{4}}}+\frac{1}{|T|^{\frac{1}{3}}c^k}+\frac{1}{|T|^{\frac{1}{3}}\sqrt{k}}\right )
$$
for some absolute constants $C$ and $c>1$. Because $T$ can be embedded in the alternating group on $p(T)$ points, it follows that $p(T)\to\infty$ as $|T|\to\infty$. In addition, we must have $k\geq 5$ since for $k\leq 4$ the only primitive permutation groups of degree $k$ are $A_k$ and $S_k$.  Thus $Q((A(k,T)\rtimes P_G),2)$ converges to 0 as $|T|\to \infty$ or $k\to \infty$. Since $G\leq A(k,T)\rtimes P_G\leq W(k,T)$, any base for $A(k,T)\rtimes P_G$ is also a base for $G$. Thus $Q(G,2)\leq Q((A(k,T)\rtimes P_G),2)$. Also, clearly
$|G|\leq |\Aut(T)||T|^{k-1}|P_G|\leq |T|!|T|^{k-1}k!,$
so $|T|\to\infty$ or $k\to\infty$ when $|G|\to\infty$. Thus $Q(G,2)$ will indeed converge to 0 as $|G|$ tends to infinity.
\end{proof}

In order to prove the three lemmas, we first need to calculate the sizes of conjugacy classes and centralisers of various elements of $D(k,T)$. 

\begin{lemma}
\label{conj eqn}
Let $P$ be a subgroup of $S_k$, let $G:=A(k,T)\rtimes P$, and let $(\alpha,\ldots,\alpha)\pi\in G$ where $\pi$ has a fixed point on $[k]$. Then
$$(\alpha,\ldots,\alpha)\pi^G\cap G_D=\{(\alpha',\ldots,\alpha')\pi':\alpha'\in \alpha^{\Aut(T)},\pi'\in\pi^P\}.$$
In particular, $|(\alpha,\ldots,\alpha)\pi^G\cap G_D|=|\alpha^{\Aut(T)}||\pi^P|$.
\end{lemma}

\begin{proof}
Suppose that $\alpha':=\beta^{-1}\alpha\beta$ for any $\beta\in \Aut(T)$ and $\pi':=\sigma^{-1}\pi\sigma$ for any $\sigma\in P$. Then $(\beta,\ldots,\beta)\sigma$ conjugates $(\alpha,\ldots,\alpha)\pi$ to $(\alpha',\ldots,\alpha')\pi'$ in $G$. On the other hand, if $(\alpha_1,\ldots,\alpha_k)\sigma$ conjugates $(\alpha,\ldots,\alpha)\pi$ to $(\alpha',\ldots,\alpha')\pi'$ in $G$, then $\sigma^{-1}\pi\sigma=\pi'$ and $\alpha_i^{-1}\alpha\alpha_{i\pi}=\alpha'$ for all $i$. Since $\pi$ has a fixed point, the result follows.
\end{proof}

The proof of Lemma \ref{conj eqn} should give the reader some indication of why it is not only convenient to work with the group $A(k,T)\rtimes P$  but also necessary, as we  lose control of the sizes of $R_2(G)$ and $R_3(G)$ for an arbitrary group of diagonal type $G$.

\begin{lemma}
\label{cc}
Let $P$ be a subgroup of $S_k$, let $G:=A(k,T)\rtimes P$, and let $(\alpha,\ldots,\alpha)\pi$ be an element of $G$ of prime order $p$. Then $|C_G((\alpha,\ldots,\alpha)\pi)|$ is either
\begin{equation}
\label{pi fpf}
|C_{P}(\pi)||C_{\Out(T)}(\overline{\alpha})||T|^{\frac{k}{p}}
\end{equation}
if $\pi$ is fixed-point-free on $[k]$, or
\begin{equation}
\label{pi not fpf}
|C_P(\pi)||C_{\Aut(T)}(\alpha)||C_{\Inn(T)}(\alpha)|^{\fix_{[k]}(\pi)-1}|T|^{\frac{1}{p}(k-\fix_{[k]}(\pi))}
\end{equation}
if $\pi$ has a fixed point on $[k]$.
\end{lemma}

Note that the division into two cases in Lemma \ref{cc} is necessary because there exist $\alpha,\beta\in Aut(T)$ for which $\overline{\beta}\in C_{\Out(T)}(\overline{\alpha})$  but $\beta\notin C_{\Aut(T)}(\alpha)$. In fact, if $\pi$ is fixed-point-free, then the two versions given of $|C_G((\alpha,\ldots,\alpha)\pi)|$ agree precisely when $C_{\Aut(T)}(\alpha)/C_{\Inn(T)}(\alpha)$ is isomorphic to $C_{\Out(T)}(\overline{\alpha})$.

\begin{proof} Let $f_\pi:=\fix_{[k]}(\pi)$, let $c_\pi$ be the number of non-trivial cycles of $\pi$ so that $c_\pi=(k-f_\pi)/p$, and let $r_\pi:=c_\pi+f_\pi$. The element $(\alpha_1,\ldots,\alpha_k)\sigma\in \Aut(T)^k\wr S_k$ is in $G$ and centralises $(\alpha,\ldots,\alpha)\pi$ if and only if all three of the following conditions occur: $\sigma$ centralises $\pi$ in $P$, $\alpha^{-1}\alpha_i\alpha=\alpha_{i\pi}$ for all $i$, and $\alpha_i$ and $\alpha_j$ are in the same coset of $\Inn(T)$ for all $i$ and $j$. There are precisely $|C_P(\pi)|$ elements of $P$ satisfying the first condition, and this condition is independent from the other two, so we assume that $\sigma$ is fixed and count how many occurrences of the latter conditions are possible.

Note that if $\alpha$ is trivial, then $\alpha^{-1}\alpha_i\alpha=\alpha_{i\pi}$ for all $i$ if and only if $\alpha_i=\alpha_j$ whenever $i$ and $j$ are in the same cycle of the full cycle decomposition of $\pi$. Thus there are $|\Out(T)||T|^{r_\pi}$ tuples $(\alpha_1,\ldots,\alpha_k)$ satisfying both conditions. The desired equality then follows in either case for $\pi$, so we may assume that $\alpha$ is non-trivial, in which case $\alpha$ has prime order $p$.

Suppose, first of all, that $i_0$ is moved by $\pi$.  Then $i_0$ is contained in a $p$-cycle in the full cycle decomposition of $\pi$ as $\pi$ must have the same prime order as $\alpha$. Let us assume that this $p$-cycle is $(1 2 \cdots p)$ and that $i_0$ is 1. If  $\alpha^{-1}\alpha_i\alpha =\alpha_{i\pi}$ and $\overline{\alpha}_1=\overline{\alpha}_i$ for all $i$, then, in particular, $[\alpha_1,\alpha]\in \Inn(T)$ and the elements  $\alpha_2,\ldots,\alpha_p$ are determined by $\alpha_1$ and $\alpha$. Conversely, if we are given $\alpha_1\in \Aut(T)$ such that $[\alpha_1,\alpha]\in \Inn(T)$, define $\alpha_{i+1}:=\alpha^{-i}\alpha_1\alpha^i$ for each $i\in[p-1]$. Then $\alpha^{-1}\alpha_i\alpha=\alpha_{i\pi}$ for all $i\in [p]$ since $\alpha$ has order $p$. Moreover, $[\alpha_1,\alpha^i]\in \Inn(T)$ for all $i\in [p]$ since $[\alpha_1,\alpha^i]=[\alpha_1,\alpha^{i-1}]\alpha^{1-i}[\alpha_1,\alpha]\alpha^{i-1}$ for all such $i$; thus $\overline{\alpha}_1=\overline{\alpha}_i$ for all $i\in [p]$. Since this argument does not depend on the choice of $i_0$ or on the letters of the $p$-cycle, and since for $\beta\in \Aut(T)$, $[\beta,\alpha]\in \Inn(T)$ if and only if $\overline{\beta}\in C_{\Out(T)}(\overline{\alpha})$, it follows that there are at most $|C_{\Out(T)}(\overline{\alpha})|$ choices for the coset of $\Inn(T)$ from which the $\alpha_i$ may be chosen, and for each such coset there are at most $|T|$ choices corresponding to each non-trivial cycle of $\pi$. If $\pi$ is fixed-point-free, then all of these choices are possible. Since $c_\pi=k/p$, equation (\ref{pi fpf}) follows.

Suppose then that $\pi$ has a fixed point $i_0$. Certainly $\alpha^{-1}\alpha_i\alpha=\alpha_{i\pi}$ and $\overline{\alpha}_{i_0}=\overline{\alpha}_i$ for all fixed points $i$ if and only if $\alpha_{i_0}\in C_{\Aut(T)}(\alpha)$ and $\alpha_{i_0}^{-1}\alpha_i\in C_{\Inn(T)}(\alpha)$ for all fixed points $i\neq i_0$. Hence there are at most $|C_{\Aut(T)}(\alpha)||C_{\Inn(T)}(\alpha)|^{f_\pi-1}$ choices for $\{\alpha_i:i\pi=i\}$, and if $\pi$ is trivial, then all of these choices are possible, in which case equation (\ref{pi not fpf}) is true. Suppose that $\pi\neq 1$, and let $\{\alpha_i:i\pi=i\}$ be one of the choices described above. Then since $\overline{\alpha}_{i_0}\in C_{\Out(T)}(\overline{\alpha})$ for any $i_0$ fixed by $\pi$, any element of the coset $\overline{\alpha}_{i_0}$ may be chosen to determine the $\alpha_j$ corresponding to any non-trivial cycle of $\pi$ as above. Thus each of the choices for $\{\alpha_i:i\pi=i\}$ not only occurs but does so $|T|^{c_\pi}$ times. Equation (\ref{pi not fpf}) then follows.
\end{proof}

There are several occasions when we will need to bound the number of conjugacy classes of elements of prime order in a group, so we set up some notation for this. Let $X$ be a group. If $\mathcal{C}$ is a union of conjugacy classes of $X$, we write $f_\mathcal{C}(X)$ for the number of conjugacy classes contained in $\mathcal{C}$. Also, we write $f(X)$ for $f_X(X)$, and when $\mathcal{C}$ consists of the elements of prime order in $X$, we write $f_p(X)$ for $f_\mathcal{C}(X)$. Let $Y$ be a subgroup of $X$. Gallagher noted in \cite{gal} that $f(X)\leq [X:Y]f(Y)$ and $f(Y)\leq [X:Y]f(X)$ and gave elementary proofs of these facts. The latter can easily be generalised to $f_\mathcal{C}(Y)$ for any union of conjugacy classes $\mathcal{C}$ in $Y$, which we now do.

\begin{lemma}
 \label{f(X)}
Let $X$ be a group with subgroup $Y$. Let $\mathcal{C}\subseteq \mathcal{C}'$ be unions of conjugacy classes of $Y$ and $X$ respectively. Then
$f_\mathcal{C}(Y)\leq [X:Y]f_{\mathcal{C}'}(X).$
In particular, $f_p(Y)\leq [X:Y]f_p(X)$.
\end{lemma}

\begin{proof}
We adapt Gallagher's proof in \cite{gal} as follows. First we obtain a formula for $f_\mathcal{C}(Y)$:
$$\frac{1}{|Y|}\displaystyle\sum_{y\in \mathcal{C}}|C_Y(y)|= \displaystyle\sum_{y\in \mathcal{C}}\frac{1}{|y^Y|}=f_\mathcal{C}(Y). $$
Of course, this formula can be used to determine $f_\mathcal{C'}(X)$ as well. Since  $\mathcal{C}\subseteq \mathcal{C}'$ and $C_Y(y)\leq C_X(y)$ for all $y\in \mathcal{C}$, the result follows.
\end{proof}

We will also need the following technical consequences of the CFSG. The first is routine to verify since $|\Out(T)|$ and $|T|$ are known for every  simple group $T$ (see \cite{kl}, for example, for lists of these quantities), so a proof is omitted. For the second, note that $l(T)$ denotes the untwisted Lie rank of a simple group $T$ of Lie type; when $T$ is a twisted group, this is simply the Lie rank of the corresponding untwisted group. Recall that $p(T)$ denotes the minimal index of a proper subgroup of $T$.

\begin{lemma}
 \label{Out(T)}
Let $T$ be a non-abelian simple group. Then $|\Out(T)|^3<|T|$. 	
\end{lemma}

\begin{lemma}
\label{p(T) bound} 
Let $T$ be a simple group of Lie type over $\mathbb{F}_q$ where $T\neq L_m(2)$ for any $m$. Then $$|\Out(T)|^2(6q)^{l(T)}\leq C p(T)^{11/4}$$ for some absolute constant $C$.
\end{lemma}

\begin{proof}
Note that the presence of the absolute constant $C$ allows us to ignore finitely many $T$. Write $q=p^f$ where $p$ is a prime and $f$ is a positive integer. Values for $|\Out(T)|$ may be found in \cite{kl}, for example.

Suppose that $T$ is an exceptional group. Since $l(T)$ is constant and $|\Out(T)|$ is bounded above by a constant multiple of $q$, it suffices to show that $l(T)+2$ is at most $(11/4)b(T)$ for some constant $b(T)$ for which $p(T)\geq q^{b(T)}$. If $T$ is $^2B_2(q)$ or $^2G_2(q)$, then $l(T)= 2$ and we may take $b(T)=2$ by \cite{v3}. Otherwise, we have $l(T)\leq 8$ and we may take $b(T)=4$ by \cite{v1,v2,v3}. In both cases, the desired inequality is satisfied.

Let $T$ be one of the following groups: $PSp_{2m}(q)$ where $m\geq 2$, $\Omega_{2m+1}(q)$ where $m\geq 3$, $P\Omega_{2m}^+(q)$ where $m\geq 4$, or $P\Omega_{2m}^-(q)$ where $m\geq 4$. Then $l(T)=m$, $p(T)\geq q^{2m-2}$ by \cite{maz2,vm1}, and $|\Out(T)|$ is at most a constant multiple of $q$. Since
$q^2(6q)^{m}$ is at most $36q^{2(2m-2)},$
it follows that $T$ satisfies the desired inequality. 

Let $T$ be $U_m(q)$ where $m\geq 3$. Then $l(T)=m-1$, $p(T)\geq q^{2m-4}$ by \cite{maz2}, and $|\Out(T)|$ is at most a constant multiple of $(q+1)f$. Since $(q+1)^2f^2\leq q^{7/2}$ and $q^{7/2}(6q)^{m-1}\leq 36q^{(11/4)(2m-4)}$, we have verified the desired inequality.

Finally, suppose that $T$ is $L_m(q)$ where $m\geq 2$. We may assume that $T\neq L_2(9)$. Then $l(T)=m-1$, $p(T)\geq q^{m-1}$ by \cite{maz2}, and $|\Out(T)|$ is at most a constant multiple of $(q-1)f$. Note that $(q-1)^2f^2\leq q^{7/2}$. If $m\geq 3$, then since $q\geq 3$ (by assumption), it follows that $q^{7/2}(6q)^{m-1}$ is at most $ 36q^{(11/4)(m-1)}$, so $T$ satisfies the desired inequality. If $m=2$, then $|\Out(T)|$ is at most a constant multiple of $f$, and since 
$f^2q\leq    q^{11/4},$
the proof is complete.  
\end{proof}

We are now in a position to prove the three lemmas.

\begin{proof}[Proof of Lemma \ref{r_1(G)}]
If $\overrightarrow{\alpha}\pi\in R_1(G)$ where $\pi$ has prime order $p$, then $k/p$ is an integer and is therefore bounded above by $ \lfloor k/2 \rfloor$. Since $P=P_G$, equation (\ref{pi fpf}) of Lemma \ref{cc} then implies that
$$\max_{\overrightarrow{\alpha}\pi\in R_1(G)}|C_G(\overrightarrow{\alpha}\pi)|\leq |\Out(T)||P||T|^{\left\lfloor \frac{k}{2}\right\rfloor} .$$
Note that $|G_D|=|\Out(T)||T||P|$ and $|G|=|G_D||T|^{k-1}$. Then
$$
r_1(G)\leq \frac{|G_D|^2}{|G|\ }\max_{\overrightarrow{\alpha}\pi\in R_1(G)}|C_G(\overrightarrow{\alpha}\pi)|\leq \frac{|\Out(T)|^2|P|^2}{|T|^{\left\lceil \frac{k}{2}\right\rceil-2}}.
$$
	
By a classification-free result of Praeger and Saxl  \cite{ps}, since $P$ is primitive and does not contain $A_k$, we know that the order of $P$ is bounded above by $4^k$. Moreover, we have $|\Out(T)|^2\leq |T|^{2/3}$ by Lemma \ref{Out(T)}. Recall that $k\geq 5$, for if $k\leq 4$ then the primitivity of $P$ implies that $P$ is $A_k$ or $S_k$; thus $\left\lceil k/2\right\rceil-17/6$ is positive. Suppose that $T$ is not $A_5$ or $L_2(7)$. Then $|T|\geq 360$, so
$$\frac{|\Out(T)|^2|P|^2}{|T|^{\left \lceil \frac{k}{2}\right\rceil-2}}\leq  \frac{16^k}{|T|^{\frac{1}{6}}360^{\left \lceil \frac{k}{2}\right\rceil-\frac{17}{6}}}\leq \frac{360^{\frac{17}{6}}}{|T|^{\frac{1}{6}}} \left(\frac{16}{\sqrt{360}}\right)^k,$$
which is our desired bound. Furthermore, by \cite[Corollary 1.2]{bab2}, which is a classification-free result of Babai, again since $P$ is primitive and does not contain $A_k$, we know that  $|P|\leq \exp(4\sqrt{k}(\log{k})^2)$ for sufficiently large $k$. Note that $k$ is eventually larger than $8\sqrt{k}(\log{k})^2$.  Suppose that $T$ is $A_5$ or $L_2(7)$. Then $|\Out(T)|=2$, so
$$\frac{|\Out(T)|^2|P|^2}{|T|^{\left \lceil \frac{k}{2}\right\rceil-2}}\leq \frac{4e^{8\sqrt{k}(\log{k})^2}}{|T|^{\frac{1}{6}} 60^{\left \lceil \frac{k}{2}\right\rceil-\frac{13}{6}}}\leq\frac{4\cdot 60^{\frac{13}{6}}}{|T|^{\frac{1}{6}}} \left(\frac{e}{\sqrt{60}}\right)^k$$
for sufficiently large $k$. Since only finitely many $G$ have been omitted from our argument, the proof is complete.
\end{proof}

\begin{proof}[Proof of Lemma \ref{r_2(G)}]
  Note that if $R(T)$ is a set of representatives for the conjugacy classes of  elements of prime order in $\Aut(T)$, then we may assume that $R_2(G)=\{\overrightarrow{\alpha}:\alpha\in R(T)\}$ by Lemma \ref{conj eqn}. By applying Lemma \ref{conj eqn} and equation (\ref{pi not fpf}) of Lemma \ref{cc} with $\pi=1$, we get the following.
$$
\begin{array}{rl}
|G|r_2(G)=&\displaystyle\sum_{\alpha\in R(T)}  \left(  |\Aut(T)|^2|C_{\Aut(T)}(\alpha)|^{-2}\right )\left(|P||C_{\Aut(T)}(\alpha)||C_{\Inn(T)}(\alpha)|^{k-1}\right) \\
\leq & |\Aut(T)|^2|P| \displaystyle\sum_{\alpha\in R(T)} |C_{\Inn(T)}(\alpha)|^{k-2}\\
\leq & |\Out(T)|^2|T|^2|P|f_p(\Aut(T))\left(\displaystyle\max_{\alpha\in R(T)} |C_{\Inn(T)}(\alpha)|\right)^{k-2}.\\
\end{array}
$$
(Recall that $f_p(X)$ denotes the number of conjugacy classes of elements of prime order in a group $X$.)
Since  $k-2$ is positive and $[T:C_{\Inn(T)}(\alpha)]\geq p(T)$ for every $1\neq \alpha\in \Aut(T)$, if we divide by $|G|$, then we see that 
$r_2(G)$ is at most $|\Out(T)|f_p(\Aut(T))p(T)^{2-k}. $ It therefore suffices to show that 
$$|\Out(T)|f_p(\Aut(T))\leq Cp(T)^{11/4}$$
for some absolute constant $C$.  Note that we may ignore finitely many simple groups $T$ should we wish to due to the presence of the constant. In particular, we may ignore the sporadic groups. 

If $T$ is the alternating group $A_m$, then $p(T)=m$ and $\Out(T)$ is constant. In fact, we have that $\Aut(A_m)=S_m$ if $m\neq 6$, and since
$$f_p(S_m)= \displaystyle \sum_{\substack{2\leq p\leq m\\ p\ \mathrm{prime}}}\left\lfloor\frac{m}{p}\right\rfloor\leq \displaystyle \sum_{\substack{2\leq p\leq m\\ p\ \mathrm{prime}}}\frac{m}{2}\leq\frac{m^2}{2},$$
it follows that $f_p(\Aut(T))p(T)^{-11/4}$ is bounded above by $m^{-3/4}$.

Let us assume, then, that $T$ is a simple group of Lie type over $\mathbb{F}_q$. As  noted  before Lemma \ref{f(X)}, for any group $X$ and subgroup $Y$, it is elementary to show that $f(X)\leq [X:Y]f(Y)$.  (Note that Lemma \ref{f(X)} provides an upper bound on $f(Y)$ rather than on $f(X)$; see \cite{gal} for a proof of the upper bound on $f(X)$.) Hence $f_p(\Aut(T))\leq f(T)|\Out(T)|$. Moreover, from \cite[Theorem 1]{lp} we know that $f(T)\leq (6q)^{l(T)}$ where $l(T)$ is the untwisted Lie rank of $T$. If $T$ is not $L_m(2)$ for any $m$, then $|\Out(T)|^2(6q)^{l(T)}\leq C p(T)^{11/4}$ for some absolute constant $C$ by Lemma \ref{p(T) bound}, and we have verified the desired inequality. If $T$ is $L_m(2)$, then $|\Out(T)|$ is constant,  $p(T)\geq 2^{m-1}$ by \cite{maz2}, and $f(T)\leq 2^m$ by \cite[Lemma 5.9]{mr}, so the proof is complete.
\end{proof}

\begin{proof}[Proof of Lemma \ref{r_3(G)} assuming Lemma \ref{r_3(G)*}]
Note first of all that $R_3(G)$ may be empty. If so, then the result is true, so we may assume otherwise. For $\pi\in P$ of prime order $p$, as in the proof of Lemma \ref{cc}, let $f_\pi:=\fix_{[k]}(\pi)$, let $c_\pi$ be the number of non-trivial cycles of $\pi$ so that $c_\pi=(k-f_\pi)/p$, and let $r_\pi:=c_\pi+f_\pi$. Since $|C_{\Inn(T)}(\alpha)|\leq |T|$, Lemma (\ref{conj eqn}) and equation (\ref{pi not fpf}) of Lemma \ref{cc} imply that
$$\begin{array}{rl}
|G|r_3(G)\leq & \!\!\!\! \displaystyle\sum_{\overrightarrow{\alpha}\pi\in R_3(G)} \!\!\!\! |\alpha^{\Aut(T)}|^2|\pi^P|^2|C_P(\pi)||C_{\Aut(T)}(\alpha)||T|^{r_\pi-1}\\
= & |\Out(T)||P|  \!\!\!\!\displaystyle\sum_{\overrightarrow{\alpha}\pi\in R_3(G)} \!\!\!\!|\alpha^{\Aut(T)}||\pi^P||T|^{r_\pi}.\\
\end{array}$$
Let $R(T)$ denote a set of representatives for the conjugacy classes of elements of prime order in $\Aut(T)$ together with the identity, and let $R(P)$ denote a set of representatives for the conjugacy classes of elements of prime order in $P$ that fix a point of $[k]$. Then by Lemma \ref{conj eqn} we may assume without loss of generality that $R_3(G)\subseteq\{\overrightarrow{\alpha}\pi:\alpha\in R(T),\pi\in R(P)\}$, so
$$
\displaystyle\sum_{\overrightarrow{\alpha}\pi\in R_3(G)} \!\!\!\!|\alpha^{\Aut(T)}||\pi^P||T|^{r_\pi}\leq\!\!\!\!\displaystyle\sum_{\alpha \in R(T)}\!\!\!\!|\alpha^{\Aut(T)}| \!\!\!\!\displaystyle\sum_{\pi \in R(P)}\!\!\!\!|\pi^P||T|^{r_\pi}  \leq |T|^{\frac{4}{3}}\!\!\!\!\displaystyle\sum_{\pi \in R(P)}\!\!\!\!|\pi^P||T|^{r_\pi}
$$
since $|\Out(T)|\leq |T|^{1/3}$ by Lemma \ref{Out(T)}. Then the proof is complete by Lemma \ref{r_3(G)*}.
\end{proof}

\begin{proof}[Proof of Lemma \ref{r_3(G)*}]
For $\pi\in P$ of prime order $p$, as in the proof of Lemma \ref{r_3(G)}, let $f_\pi:=\fix_{[k]}(\pi)$, let $c_\pi$ be the number of non-trivial cycles of $\pi$ so that $c_\pi=(k-f_\pi)/p$, and let $r_\pi:=c_\pi+f_\pi$. As in the statement of the lemma, let $R(P)$ denote a set of representatives for the conjugacy classes of elements of prime order in $P$. We want to prove that
\begin{equation}
 \label{r3 strong}
\displaystyle\sum_{\pi \in R(P)} \frac{|\pi^P|}{|T|^{k-r_\pi-\frac{5}{3}}}\leq C\left( \frac{1}{c^k}+\frac{1}{\sqrt{k}} \right)
\end{equation}
for some absolute constants $C$ and $c>1$. Note that $r_\pi=k-1$ for some $\pi\in R(P)$ if and only if $P$ contains a transposition, which is equivalent to $P$ being $S_k$ since $P$ is primitive. Thus the exponent of $|T|$ in (\ref{r3 strong}) is always positive.

Let $\pi\in R(P)$ have order $p$. We may write $p c_\pi=\mu(P)+i$ for some non-negative integer $i$ where $\mu(P)$ denotes the minimal degree of $P$, which is the minimal number of points moved by an element of $P$. Then $c_\pi=\mu(P)/p+i/p$ and $f_\pi=k-\mu(P)-i$. Since $i/p-i\leq 0$ and $p\geq 2$, it follows that
$$\max_{\pi\in R(P)} r_\pi\leq  \left \lfloor\frac{\mu(P)}{2}\right \rfloor+k-\mu(P)=k-\left\lceil \frac{\mu(P)}{2}\right\rceil.$$
Hence we conclude that (\ref{r3 strong}) is true if the following inequality holds:
\begin{equation}
 \label{r3}
\frac{|P|}{|T|^{\left\lceil \frac{\mu(P)}{2}\right\rceil-\frac{5}{3}}}\leq C\left( \frac{1}{c^k}+\frac{1}{\sqrt{k}} \right).
\end{equation}
It will usually be sufficient to prove this inequality. The proof now divides into two cases according to whether $\mu(P)\geq k/3$ or not. In the first, we bound the left-hand side of (\ref{r3 strong}) or (\ref{r3}) by $C/c^k$, and in the second, we bound the left-hand side of  (\ref{r3 strong})  or (\ref{r3}) by $C/\sqrt{k}$.

\underline{Case 1}: $\mu(P)\geq k/3$.

Suppose first of all that $|T|\geq |L_3(3)|= 5616$ and $k>6$. Since $|P|\leq 4^k$ by \cite{ps} and $\lceil k/6 \rceil -5/3$ is positive, 
$$
\frac{|P|}{|T|^{\left\lceil \frac{\mu(P)}{2}\right\rceil-\frac{5}{3}}}\leq \frac{4^k}{5616^{\left\lceil \frac{k}{6}\right\rceil-\frac{5}{3}}}\leq 5616^{\frac{5}{3}}\left({\frac{4}{\sqrt[6]{5616}}}\right)^k,
$$
which is the upper bound we desire. Suppose instead that $|T|<5616$. For  sufficiently large $k$, we know that $|P|$ is at most $ \exp(4\sqrt{k}(\log{k})^2)$ by \cite[Corollary 1.2]{bab2}. Since $k$ is eventually larger than  $24\sqrt{k}(\log{k})^2$, it follows that
$$
\frac{|P|}{|T|^{\left\lceil \frac{\mu(P)}{2}\right\rceil-\frac{5}{3}}}\leq 60^{\frac{5}{3}}\left( \frac{e^{24\sqrt{k}(\log{k})^2}}{60^k}\right) ^{\frac{1}{6}}\leq 60^{\frac{5}{3}}\left( \frac{e}{60}\right) ^{\frac{k}{6}}
$$
for  sufficiently large $k$, which is again the upper bound we desire.  Lastly, suppose that $k=5$ or 6 (which we may do since $P$ must contain $A_k$ when $k\leq 4$).  Note that the left-hand side of (\ref{r3 strong}) is bounded above by $|P||T|^{5/3-k+r_{\pi^*}}$ where $ \pi^*\in R(P)$ achieves the maximum. Since $k-r_{\pi^*}\geq 2$, we may replace $|T|$ by 60, and since $|P|$ and $r_{\pi^*}$ are constant, this establishes equation (\ref{r3 strong}). Only finitely many $G$ have been excluded from our argument, so this case is complete.

\underline{Case 2}: $\mu(P)<k/3$.

Let $\Omega_{m,l}$ denote the set of subsets of $[m]$ of size $l$. Then by Liebeck and Saxl \cite[Theorem 2]{ls}, our assumption on $\mu(P)$ forces $P$ to be a subgroup of $S_m\wr_r S_r$ that contains $A_m^r$ and acts by the product action on $\Omega_{m,l}^r$ for some  $m\geq 5$, $r\geq 1$ and $1\leq l<m/2$. Note that  this action is primitive and faithful, that $(r,l)$ is not $(1,1)$ by assumption, and that $k=\binom{m}{l}^r$. Let
$$
g(m,r,l):=\binom{m-2}{l-1}\binom{m}{l}^{r-1}.
$$
Observe that $((1 2),1,\ldots,1)\in S_m^r$ moves $2g(m,r,l)$ points of $\Omega_{m,l}^r$ while no element of $S_m\wr_r S_r$ moves fewer; hence $g(m,r,l)\leq \mu(P)/2$. It is certainly true that $m^{mr}\geq \sqrt{k}$ and $|P|\leq m^{mr}r^r$, so since $g(m,r,l)\neq 1$ and $|T|\geq 60$, it follows that (\ref{r3}) is true if we can show that
\begin{equation}
\label{r3 (m,r,l)}
 2mr\log{m}+r\log{r}\leq g(m,r,l)\log{60}+C
\end{equation}
for some absolute constant $C$. If $r\geq 3$, then equation (\ref{r3 (m,r,l)}) holds since $g(m,r,l)\geq m^{r-1}$; if $r=2$ and $l\geq 2$, then equation (\ref{r3 (m,r,l)}) holds since $g(m,2,l)\geq m^2$; and if $r=1$ and $l\geq 3$, then equation (\ref{r3 (m,r,l)}) holds since $g(m,1,l)\geq (m-3)^2/2$ (and since $l<m/2$ forces $m>6$). Thus the cases when $(r,l)$ is $(1,2)$ or $(2,1)$ remain; note that for either one, the left-hand side of equation  (\ref{r3}) tends to infinity if $T$ is fixed and $m$ tends to infinity. We therefore establish equation (\ref{r3 strong}) instead.

Suppose that $(r,l)$ is  $(1,2)$. Recall that $P$ is $A_m$ or $S_m$ acting (faithfully) on the set $\Omega_{m,2}$ of 2-subsets of $[m]$ where $m\geq 5$. In the proof of Lemma \ref{r_2(G)}, we saw that $f_p(S_m)\leq m^2/2$, and so $f_p(P)\leq m^2$. But $m\geq \sqrt{k}$, so equation (\ref{r3 strong}) will be true if we can show that $|\pi^P|60^{r_\pi-k}$ is bounded above by $m^{-3}$ for each $\pi\in R(P)$. To this end, let $\pi$ be an element of $P$ of prime order $p$. Then the full cycle decomposition of $\pi$ in $S_m$ consists of $t$ cycles of length $p$ for some $t$ such that $1\leq t\leq \lfloor m/p \rfloor$. Certainly $|\pi^P|\leq m^{pt}$. Moreover, we have $k-r_\pi=(1-1/p)(k-f_\pi)\geq (k-f_\pi)/2$ and $\log{60}>4$, so it suffices to show that
\begin{equation}
\label{r3 (1,2)}
(pt+3)\log{m}\leq 2(k-f_\pi).
\end{equation}
Let $i$ and $j$ be distinct points of $[m]$. Clearly $\pi$ fixes $\{i,j\}$ if and only if either both $i$ and $j$ are  members of $\fix_{[m]}(\pi)$, or the full cycle decomposition of $\pi$ in $S_m$ contains the transposition $(i j)$. Hence
$$
f_\pi=|\fix_{\Omega_{m,2}}(\pi)|=\left \{
\begin{array}{ll}
 \binom{m-pt}{2} & \mbox{if}\ p\geq 3\\
 \binom{m-2t}{2}+t & \mbox{if}\ p=2.\\
\end{array}
\right.
$$
 By evaluating $2(k-f_\pi)$ and rearranging equation (\ref{r3 (1,2)}), it follows that equation (\ref{r3 strong}) is true if
\begin{equation}
\label{r3 (1,2)b}
(pt+3)\log{m}+p^2t^2+pt + \left \{
\begin{array}{ll}
 0 & \mbox{if}\ p\geq 3\\
 2t & \mbox{if}\ p=2\\
\end{array}
\right\}
\leq 2mpt
\end{equation}
for all primes $p$ and integers $t$ such that $1\leq t\leq \lfloor m/p \rfloor$. Since $pt+1\leq m+1\leq 4m/3$, it follows that $p^2t^2+pt\leq 4mpt/3$. In fact, we also have $4t^2+4t\leq 8mt/3$  since $2t+2\leq m+2\leq 4m/3$ when $m\geq 6$ and $2t\leq 4$ when $m=5$. Moreover, the fact that $3\log{m}\leq m$ implies that $(pt+3)\log{m}\leq 2mpt/3$ when $pt+3\leq 2pt$. Thus equation (\ref{r3 (1,2)b}) is satisfied if $p\neq 2$ or $t\neq 1$, and if $p=2$ and $t=1$, then it is easy to check that equation (\ref{r3 (1,2)b}) still holds.

The remaining case to consider is when $(r,l)$ is $(2,1)$. Recall that here $P$ is a subgroup of $S_m^2\rtimes C_2$ that contains $A_m^2$ and acts via the product action on $[m]^2$ for $m\geq 5$. Let $Q:=S_m^2\rtimes C_2$ and let $\tau$ denote the generator for $C_2$. First we determine the conjugacy classes of elements of prime order in $Q$. 

Let $\mathcal{C}$ be the union of those elements of prime order in $Q$ whose projection onto $C_2$ is trivial, and let $\mathcal{C}_\tau$ be the union of those elements of prime order in $Q$ whose projection onto $C_2$ is $\tau$. Then the elements in $\mathcal{C}$ with order $p$ have the form $(s_1,s_2)$ where $s_1$ and $s_2$ are elements of $S_m$ such that $s_i^p=1$ for both $i$ and $s_1$ or $s_2$ is non-trivial, and the elements of $\mathcal{C}_\tau$ have the form $(s,s^{-1})\tau$ for any $s\in S_m$. Note that both $\mathcal{C}$ and $\mathcal{C}_\tau$ are unions of conjugacy classes of $Q$ since $S_m^2\unlhd Q$. In fact, since $(s,u)^{-1}(s,s^{-1})\tau(s,u)=(u,u^{-1})\tau$ for any $s,u\in S_m$, it follows that $f_{\mathcal{C}_\tau}(Q)=1$. 

Let $(s_1,s_2)\in \mathcal{C}$. Since we may conjugate $(s_1,s_2)$ by $(u_1,u_2)$ or $(u_1,u_2)\tau$ for any $u_1,u_2\in S_m$, it follows that 
$(s_1,s_2)^Q=(s_1^{S_m}\times s_2^{S_m})\cup (s_2^{S_m}\times s_1^{S_m}).$
Fix a prime $p\leq m$. Then in $Q$ there are $m_p:=\lfloor m/p\rfloor$ conjugacy classes $(s_1,s_1)^Q$ where $s_1$ has order $p$, and since $(s_1,s_2)^Q=(s_2,s_1)^Q$, there are $\binom{m_p+1}{2}$ conjugacy classes $(s_1,s_2)^Q$ where $(s_1,s_2)$ has order $p$ but $s_1$ and $s_2$ have a different number of $p$-cycles on $[m]$ (allowing for the identity, which has no $p$-cycles). This accounts for all of the elements in $\mathcal{C}$ with order $p$. Then since $m_p\leq m/2$ for any prime $p$, we obtain
$$f_{\mathcal{C}}(Q)=\displaystyle \sum_{\substack{2\leq p\leq m\\ p\ \mathrm{prime}}}m_p+\binom{m_p+1}{2}\leq \displaystyle \sum_{\substack{2\leq p\leq m\\ p\ \mathrm{prime}}}\frac{m^2+6m}{8}\leq \frac{m^3+6m^2}{8}.$$
Thus $f_{\mathcal{C}}(Q)\leq (3/8)m^3$.

Since $P$ has index at most 8 in $Q$,  Lemma \ref{f(X)} implies that  $f_{\mathcal{C}\cap P}(P)\leq 3m^3$ and that $f_{\mathcal{C}_\tau\cap P}(P)\leq 8$. But $m=\sqrt{k}$, so equation (\ref{r3 strong}) is true if $|\pi^P|60^{r_\pi-k}$ is at most a constant multiple of  $m^{-4}$ for all $\pi\in R(P)\cap \mathcal{C}$ and at  most a constant multiple of  $m^{-1}$ for all $\pi\in R(P)\cap \mathcal{C}_\tau$  where both constants are absolute. 

We prove the latter requirement first. Let $\pi\in \mathcal{C}_\tau\cap P$. Then $|\pi^P|\leq m^{m-1}$ since $|\mathcal{C}_\tau|=|S_m|$. Moreover, if $\pi=(s,s^{-1})\tau$, then the set of fixed points of $\pi$ on $[m]^2$ is $\{(i,is):i\in [m]\}$, so $2(k-r_{\pi})=k-f_{\pi}=m^2-m$. Since $\log{m}\leq \log{60}(m-1)/2$, we have that $|\pi^P|60^{r_\pi-k}$ is bounded above by $m^{-1}$, as desired.

Now let $\pi=(s_1,s_2)$ be an element of prime order $p$ in $\mathcal{C}\cap P$, and suppose that for each $i$ the full cycle decomposition of $s_i$ in $S_m$ consists of $t_i$ $p$-cycles where $0\leq t_i\leq \lfloor m/p \rfloor$ and $t_1$ or $t_2$ is non-zero. Then $|(s_1,s_2)^P|\leq 2|s_1^{S_m}||s_2^{S_m}|\leq 2m^{pt_1+pt_2}$. Moreover, the element $(s_1,s_2)$ fixes $(i,j)\in [m]^2$ if and only if $s_1$ fixes $i$ and $s_2$ fixes $j$, so $f_{(s_1,s_2)}=(m-pt_1)(m-pt_2)$. Again, since $k-r_\pi\geq (k-f_\pi)/2$ and $\log{60}>4$, if we can show that
$$
(pt_1+pt_2+4)\log{m} + 2p^2t_1t_2\leq 2m(pt_1+pt_2),
$$
then $|\pi^P|60^{r_\pi-k}$ is  bounded above by  $2m^{-4}$, as desired. Since $\log{m}/m$ is at most $1/3$ and $x:=pt_1+pt_2$ is at least two, we obtain that $(x+4)\log{m} \leq mx$, and since $pt_i\leq m$ for both $i$, we obtain that $2p^2t_1t_2\leq mx$. This completes the proof.
\end{proof}

Now we prove Theorem \ref{diag prob k fixed} by modifying the proofs of Lemmas \ref{r_1(G)} and \ref{r_3(G)}.

\begin{proof}[Proof of Theorem \ref{diag prob k fixed}]
As in the proof of Theorem \ref{diag prob}, it suffices to show that if $G=A(k,T)\rtimes P$ where $P$ is a primitive subgroup of $S_k$ and $k\geq 5$, then $r_i(G)$ converges to 0 for each $i$   as $|T|$ tends to infinity with $k$ fixed. In the proof of Lemma \ref{r_1(G)}, we saw that $r_1(G)\leq |P|^2|T|^{8/3-\lceil k/2 \rceil}$ (since $|\Out(T)|^2<|T|^{2/3}$ by Lemma \ref{Out(T)}). But $k$ is at least 5, so $r_1(G)\to 0$ as $|T|\to \infty$ with $k$ fixed. Moreover, by Lemma \ref{r_2(G)} the same is true for $r_2(G)$ since $p(T)\to\infty$ as $|T|\to \infty$. Thus it remains to consider $r_3(G)$; this will require some extra work.

For $\pi\in P$ of prime order $p$, as we have done before, let $f_\pi:=\fix_{[k]}(\pi)$, $c_\pi:=(k-f_\pi)/p$ and $r_\pi:=c_\pi+f_\pi$. Also, let $R(P)$ denote a set of representatives for the conjugacy classes of elements of prime order in $P$ that also fix a point. By Lemma \ref{conj eqn}, we may assume that if  $\overrightarrow{\alpha}\pi\in R_3(G)$, then $\pi\in R(P)$; moreover, we will assume for simplicity that if $\pi$ is a transposition, then $\pi=(1 2)$. Accordingly, let $R_4(G):=\{\overrightarrow{\alpha}\pi\in R_3(G):\pi =(1 2)\}$, and let $r_4(G)$ be the sum of $r_3(G)$ restricted to elements of $R_4(G)$. Also, let $R_4(T):=\{\alpha\in \Aut(T): \overrightarrow{\alpha}\pi\in R_4(G) \}$. Note that $R_4(T)$ contains the identity if it is non-empty. Suppose for the time being that $R_4(G)$ is non-empty. Then $P=S_k$. By Lemma (\ref{conj eqn}) and equation (\ref{pi not fpf}) of Lemma \ref{cc}, we have that
$$\begin{array}{rl}
r_4(G)=  &  |(1 2)^{S_k}| \displaystyle\sum_{\alpha\in R_4(T)} \frac{|\alpha^{\Aut(T)}|}{|T|}\left (\frac{|C_{\Inn(T)}(\alpha)|}{|T|}\right)^{k-3}\\
 \leq & |(1 2)^{S_k}| \left( \frac{1}{|T|} +  \frac{|\Out(T)|}{p(T)^{k-3}}\right)\\
\end{array}$$
since $[T:C_{\Inn(T)}(\alpha)]\geq p(T)$ if $\alpha\neq 1$. Then for any primitive $P$, the proof of Lemma \ref{r_3(G)} and the above inequality imply that
$$r_3(G)\leq |(1 2)^{S_k}| \left( \frac{1}{|T|} +  \frac{|\Out(T)|}{p(T)^{k-3}}\right) +\displaystyle \sum_{\pi\in R(P)\setminus\{(1 2)\}}\frac{|\pi^P|}{|T|^{k-r_\pi-\frac{5}{3}}}.$$
But $|\Out(T)|\leq Cp(T)^{11/8}$ for some absolute constant $C$ by Lemma \ref{p(T) bound} since $|\Out(T)|$ is constant if $T$ is either $L_m(2)$, an  alternating group or a sporadic group. This then gives us the desired convergence since  $k\geq 5$ and $k-r_\pi\geq 2$ when $\pi$ is not a transposition.
\end{proof}

Note that the methods of this section can sometimes be adapted to the cases when either $k<5$ and $k$ is fixed as $|G|\to \infty$, or $k<|T|$ and $k$ grows with $|T|$ as $|G|\to \infty$. These results and their proofs are not included here but can be found in \cite{faw}.  For example, it is proved  that if $k\geq 3$, then the proportion of 3-tuples that are bases for $G$ tends to 1 if $k$ is fixed as $|G|\to\infty$. In fact, it looks likely that a similar result occurs when $k=2$ and $P_G$ is trivial; indeed, it is proved  for such groups that the proportion of 4-tuples that are bases tends to 1 if $k$ is fixed as $|G|\to\infty$ by using stronger bounds on the numbers of conjugacy classes of non-abelian simple groups. Lastly, it is proved that the proportion of pairs that are bases tends to 1 as $|G|$ tends to infinity if $k$ is a non-constant function in the variable $|T|$ for which $k^4\leq |T|$ for all non-abelian simple groups $T$.

\begin{ack}
I am grateful to Jan Saxl for his support and guidance, as well as to Ross Lawther and Peter Cameron for some helpful suggestions.
\end{ack}

\bibliographystyle{acm}
\bibliography{jbf_references}

\end{document}